\documentclass[15pt]{article}

\usepackage{amsmath,amsfonts,amssymb,amsthm,amscd}
\newcounter{stepctr}
{\end{list}}

\newtheorem{thm}{Theorem}[section]
\newtheorem{prop}[thm]{Proposition}
\newtheorem{cor}[thm]{Corollary}

\theoremstyle{definition}
\newtheorem{dfn}[thm]{Definition}
\newtheorem{ex}[thm]{Example}

\newtheorem{rema}[thm]{Remark}
\newtheorem{lem}[thm]{Lemma}
\newtheorem{prob*}{Open problem}
\newcommand{\demo}{\begin{proof}}

\newcommand{\NN}{\mathcal{N}}
\newcommand{\R}{\ensuremath{\mathcal{R}}}

\newcommand{\N}{\mathbb{N}}

\def\ll^2{{\mathcal L}(\ell^2(\N))}

\def\f^0x{{\mathcal F^0}(X) }

\title
{\bf  On the class  $(W_{e})$-operators}

\author{   Zakariae Aznay, Hassan Zariouh }

\date{}
\begin{document}

\maketitle \thispagestyle{empty}

\begin{abstract}\noindent\baselineskip=10pt
It is well known that an hyponormal operator satisfies  Weyl's theorem. A result due to Conway shows that  the essential spectrum  of a normal operator $N$   consists precisely of all points in its spectrum except the isolated eigenvalues of finite multiplicity, that's $\sigma_{e}(N)=\sigma(N)\setminus E^0(N).$ In this paper, we define and study a new class  named  $(W_{e})$ of operators satisfying $\sigma_{e}(T)=\sigma(T)\setminus E^0(T),$ as a subclass of $(W).$  A countrexample shows  generally that an hyponormal  does not belong   to the class $(W_{e}),$    and we give an additional hypothesis under which an hyponormal  belongs to the class  $(W_{e}).$     We also give  the generalisation class $(gW_{e})$ in the contexte of B-Fredholm theory, and  we  characterize      $(B_{e}),$ as a subclass of $(B),$     in terms of localized SVEP.

\end{abstract}

 \baselineskip=15pt
  \baselineskip=15pt

\section{Introduction and preliminaries}

Let $X$ denote an infinite dimensional complex Banach space, and
denote by $L(X)$ the algebra of all bounded linear operators on $X.$
For $T\in L(X),$ we  denote by $T^*,$ $\alpha(T)$  and   $\beta(T)$ the dual of $T,$ the dimension of the
kernel $\mathcal{N}(T)$  and  the codimension of the range $\R(T),$ respectively.
By $\sigma (T),$ $\sigma_a(T)$ and  $\sigma_s(T),$  we denote the spectrum,  
 the approximate spectrum  and the surjectivity spectrum  of $T,$ 
respectively.\\
 Recall that $T$ is said to be \textit{upper semi-Fredholm}, if $\R(T)$ is
closed and $\alpha(T) <\infty,$ while $T$ is called \textit{lower
semi-Fredholm}, if   $\beta(T) < \infty.$ $T\in
L(X)$ is said to be \textit{semi-Fredholm} if $T$ is either an upper
semi-Fredholm or a lower semi-Fredholm operator. $T$ is
\textit{Fredholm} if $T$ is upper semi-Fredholm and lower
semi-Fredholm. If $T$ is semi-Fredholm then the index of $T$ is
defined by $\mbox{ind}(T)=\alpha(T) -\beta(T).$ For an operator $T \in
L(X),$ the \textit{ascent} $p(T)$ and the \textit{descent} $q(T)$
are defined by $p(T) = \inf\{n\in \mathbb{N}: \mathcal{N}(T^n) = \mathcal{N}(T^{n+1})\}$
 and $q(T)= \inf\{n\in \mathbb{N}: \R(T^n) = \R(T^{n+1})\},$
 respectively; the
infimum over the empty set is taken $\infty.$ If the ascent and the
descent of $T$  are both finite, then $p(T) = q(T)=p,$  and $\R(T^p)$ is closed. An operator $T$ is said to be \textit{Weyl} if it is
Fredholm of index zero, and is said to be 
 \textit{upper semi-Browder} if it is an upper
semi-Fredholm operator with finite ascent and it is called
\textit{Browder} if it is Fredholm of finite
ascent and descent. \\
If $T\in L(X)$ and $ n \in \mathbb{N},$
 we denote  by $T_{[n]}$ the restriction
of  $T$ on $\R(T^n).$ $T$ is said to be \textit{semi-B-Fredholm}  if there
exists $n\in \mathbb{N}$ such that
$\R(T^n) $ is closed and $T_{[n]}:
\R(T^n)\rightarrow \R(T^n)$ is semi-Fredholm. A B-Fredholm (resp., B-Weyl) operator is similarly defined.

 We recall that a complex number $\lambda\in\sigma (T)$ is a
\textit{pole} of the resolvent of $T,$ if $T-\lambda I$ has finite
ascent and finite descent, and $\lambda\in\sigma_a(T)$ is a
\textit{left pole} of $T$ if $p=p(T-\lambda I) <\infty$ and
$\R(T^{p+1})$ is closed.\\
In the following  list, we summarize  the notations and symbols
needed later.

\smallskip
\noindent $\mbox{iso}\,A$: isolated points in a given subset A of $ \mathbb{C},$\\
\noindent $\mbox{acc}\,A$: accumulation points of a given subset A of $ \mathbb{C},$\\
 \noindent $A^C$: the complementary of a subset $A\subset \mathbb{C},$\\
 \noindent $(W)$: the class of operators satisfying Weyl's theorem [$T\in(W)$ if $\Delta(T)=E^0(T)$],\\
\noindent $(gW)$: the class of operators satisfying generalized Weyl's theorem [$T\in(gW)$ if $\Delta^g(T)=E(T)$],\\
\noindent $(B)$: the class of operators satisfying Browder's theorem [$T\in(B)$ if $\Delta(T)=\Pi^0(T)$],\\
\noindent $(gB)$: the class of operators satisfying generalized Browder's theorem [$T\in(gB)$ if $\Delta^g(T)=\Pi(T)$],\\
\noindent $(aB)$: the class of operators satisfying a-Browder's theorem [$T\in(aB)$ if $\Delta_{a}(T)=\Pi_a^0(T)$].\\

 \begin{tabular}{l|l}
 $\sigma_{b}(T)$:  Browder spectrum of $T$ &  $\Delta(T):=\sigma(T)\setminus\sigma_{w}(T)$\\
 $\sigma_{ub}(T)$:  upper semi-Browder spectrum of $T$ & $\Delta^g(T):=\sigma(T)\setminus\sigma_{bw}(T)$\\
 $\sigma_{w}(T)$: Weyl spectrum of $T$ & $\Delta_{e}(T):=\sigma(T)\setminus\sigma_{e}(T)$\\
 $\sigma_{uw}(T)$: upper semi-Weyl spectrum of $T$ & $\Delta_{e}^g(T):=\sigma(T)\setminus\sigma_{bw}(T)$\\
 $\sigma_{e}(T)$: essential spectrum of $T$ &  $\Delta_{a}(T):=\sigma_{a}(T)\setminus\sigma_{uw}(T)$\\
 $ \sigma_{d}(T)$:  Drazin spectrum of $T$ &  $\Pi^0(T)$: poles of $T$ of finite rank\\
  $\sigma_{ld}(T)$: left Drazin spectrum of $T$ &  $\Pi(T)$: poles of $T$\\
$\sigma_{bw}(T)$: B-Weyl spectrum of $T$ &  $\Pi_a^0(T)$: left  poles of $T$ of finite rank\\
$\sigma_{bf}(T)$: B-Fredholm  spectrum of $T$ &  $E^0(T):=\mbox{iso}\,\sigma(T)\cap\sigma_{p}^0(T)$\\
  $\sigma_{p}(T)$:  eigenvalues of $T$  & $E(T):=\mbox{iso}\,\sigma(T)\cap\sigma_{p}(T)$\\
 $\sigma_{p}^0(T)$: eigenvalues of $T$ of finite multiplicity    &  $\sigma_{uf}(T)$: upper semi-Fredholm spectrum of $T$\\
 $\sigma_{rd}(T)$:  right Drazin spectrum of $T$    &  $\sigma_{lf}(T)$: lower semi-Fredholm spectrum of $T$\\
 $\sigma_{ubf}(T)$: upper semi-B-Fredholm spectrum of $T$    &  $\sigma_{lbf}(T)$: lower semi-B-Fredholm spectrum of $T$\\
 $\sigma_{ubw}(T)$: upper semi-B-Weyl spectrum of $T$    &  $\sigma_{lbw}(T)$: lower semi-B-Weyl spectrum of $T$ \\
   
\end{tabular}\\

For more details on several classes and spectra originating from Fredholm theory or B-Fredholm, we refer the reader to \cite{aiena1, Berkani-koliha}.

For the sake of completeness, we include the following  corollaries as straightforward applications of \cite[Theorem 4.7]{Grabiner}. For definitions and properties of operators with topological uniform descent, we refer the reader to the paper of Grabiner \cite{Grabiner}.

 \begin{cor}\label{corollary0}{\rm\cite[Corollary 4.8]{Grabiner}}  Suppose that $T\in L(X)$ is a bounded linear operator with topological uniform descent for $n\geq d,$ and that $V$ is a bounded  linear operator commuting with $T$. If $V-T$ is invertible and sufficiently small, then\\
 (i) $V$ has infinite ascent or descent if and only if $T$ does.\\
 (ii) $V$ cannot have finite non-zero ascent or descent.\\
 (iii) $V$ is onto if and only if $T$ has finite descent.\\
 (iv) $V$ is one-to-one (or bounded below) if and only if $T$ has finite ascent.\\
 (v) $V$ is invertible if and only if $0$ is a pole of $T.$\\
 (vi) $V$ is upper semi-Fredholm if and only if some $\frac{\mathcal{N}(T^{n+1})}{\mathcal{N}(T^{n})}$ is finite-dimensional.\\
 (vii) $V$ is lower semi-Fredholm if and only if some $\frac{\R(T^{n})}{\R(T^{n+1})}$ is finite-dimensional.
 \end{cor}

 \begin{cor}\label{corollary1}{\rm\cite[Corollary 3.2]{berkani-sarih}} Suppose that $T\in L(X)$ is a bounded linear with topological uniform descent for $n\geq d,$ and that $V$ is a bounded linear operator commuting with $T$. If $V-T$ is invertible and sufficiently small, then the following hold.\\
 (i) If $\dim(\mathcal{N}(T)\cap\R(T^d))<\infty,$  then $V$  is an upper semi-Fredholm  and $\alpha(V)=\dim(\mathcal{N}(T)\cap\R(T^d)).$\\
 (ii) If ${\rm codim}(\R(T)+\mathcal{N}(T^d))<\infty,$ then $V$  is a lower semi-Fredholm   and $\beta(V)={\rm codim}(\R(T)+\mathcal{N}(T^d)).$
 \end{cor}

\section{ On the class  $(B_{e})$-operators}

We begin this section by the following definition, in which we introduce  new class of operators named $(B_{e})$-operators,  as a proper subclass of the class $(B)$ of operators satisfying Browder's theorem (see Proposition \ref{propB} below).

\begin{dfn}\label{dfn1}A bounded linear operator $T\in L(X)$ is said to belong to the class $(B_{e})$ [$T\in (B_{e})$ for brevity] if  $\Delta_{e}(T)=\Pi^0(T)$ and is said to belong to the class
 $(gB_{e})$ [$T\in (gB_{e})$ for brevity] if $\Delta_{e}^g(T)=\Pi(T).$ In other words,  $T\in (B_{e})$ if and only if  $\sigma_{e}(T)=\sigma_{b}(T)$ and $T\in (gB_{e})$ if and only if $\sigma_{bf}(T)=\sigma_{d}(T).$
\end{dfn}

Before giving some examples of operators which illustrate the classes  $(B_{e})$ and $(gB_{e}),$ we recall the  property of SVEP which will be useful in everything that follows. A bounded linear
 operator $T\in L(X)$ is said to have the {\it single-valued
 extension property} (SVEP for short) at $\lambda\in\mathbb{C}$ if
  for every open neighborhood $U_\lambda$ of $\lambda,$ the  function $f\equiv 0$ is the only
 analytic solution of the equation
 $(T-\mu I)f(\mu)=0\quad\forall\mu\in U_\lambda.$ We denote by
  ${\mathcal S}(T)=\{\lambda\in\mathbb{C}: T\mbox{  does not have the SVEP at } \lambda\}$
     and we say that  $T$ has the  SVEP
  if
$ {\mathcal S}(T)=\emptyset.$ We say that $T$ has the SVEP on $A\subset\mathbb{C},$ if $T$ has the SVEP
at every $\lambda\in A.$ (For more details about this property, we refer the reader to \cite{aiena1}). Thus it follows easily  that $T\in L(X)$ has the  SVEP at every point of the boundary $\partial\sigma(T)$ of the spectrum $\sigma(T),$ and in particular, $T$ has the SVEP at every  isolated  point in
$\sigma(T).$ We also have from  \cite[Theorem 2.65]{aiena1},  \[p(T-\lambda_0 I)<\infty \Longrightarrow
\mbox{ T has the SVEP at } \lambda_0, \,\,\,(A)\] and dually
\[q(T-\lambda_0 I)<\infty \Longrightarrow \mbox{ $T^*$ has the SVEP
at } \lambda_0.\,\,\,(B)\]
 Furthermore, if $T-\lambda_0 I$ is
semi-B-Fredholm then the implications above are equivalences.

\begin{ex}\label{ex1}~~

\noindent 1. Hereafter, the Hilbert space $ l^2(\mathbb{N})$ is  denoted by $\l^2.$ We consider the operator $T$ defined on $l^2$ by $T(x_1,x_2,x_3,\ldots)=(\frac{x_2}{2},\frac{x_3}{3}, \frac{x_4}{4},\ldots).$ It is  easily seen that  $\sigma_{e}(T)=\sigma_{b}(T)=\{0\}$ and $\sigma_{bf}(T)=\sigma_{d}(T)=\{0\}.$ This means  that  $T\in (B_{e})\cap(gB_{e}).$\\
2. Here and elsewhere,  we denote by $R$ and $L$ respectively,   the  right shift and left shift  operators    defined on  $l^2$   by
$R(x_1,x_2,x_3,\ldots)=(0,x_1,x_2,x_3,\ldots),$ and  $L(x_1,x_2,x_3,\ldots)=(x_2,x_3,\ldots).$  It is easy to get    $\sigma_{uf}(R)=\sigma_{e}(R)=C(0, 1)$ and $\sigma_{b}(R)=\sigma_{d}(R)=D(0, 1);$ where $C(0, 1)$ and $D(0, 1)$ are respectively,   the unit circle  and the closed unit disc of  $ \mathbb{C}.$ So $R\not\in (B_{e}).$ Moreover,   as  $\sigma_{ubf}(R)\subset\sigma_{bf}(R)\subset\sigma_{e}(R)$ and   $R$ has no eigenvalues, then  $\sigma_{ubf}(R)=C(0, 1).$ So   $\sigma_{bf}(R)=C(0, 1)$  and   then  $R\not\in (gB_{e}).$\\
3. Let  $T \in L(X)$ be an operator with finite accumulation points of its spectrum.  Then   $T\in(B_{e})\cap (gB_{e}).$   To see this, as ${\mathcal S}(T\oplus T^*)\subset \mbox{acc}\,\sigma(T)$  and since    ${\mathcal S}(T\oplus T^{*})$ is an  open set of $\mathbb{C},$  it follows that ${\mathcal S}(T\oplus T^{*})=\emptyset.$ On the other hand, it is easy to get  $\sigma_{d}(T)=\sigma_{bf}(T)\cup  {\mathcal S}(T\oplus T^{*}),$  this implies that   $T \in (gB_{e}).$ Furthermore, since  we always  have    $\sigma_{b}(T)= \sigma_{e}(T)\cup \sigma_{d}(T),$ then $T \in (B_{e}).$ In particular, if $T$ is a meromorphic then $T\in (B_{e})\cap(gB_{e}).$ We recall that an operator $T \in L(X)$ is called  meromorphic if $\sigma_{d}(T)\subset\{0\}$ and    note that for this type of operators, we have then $\mbox{acc}\,\sigma(T)\subset\{0\}.$
\end{ex}

\begin{rema}\label{rema0}
 It is well known that  $\sigma_{e}(T)=\sigma_{e}(T^*),$  $\sigma_{b}(T ) = \sigma_{b}(T^*),$ $\sigma_{bf}(T ) = \sigma_{bf}(T^*)$  and $\sigma_{d}(T ) = \sigma_{d}(T^*).$ Thus $T\in (B_{e})$  [resp., $(gB_{e})$] if  and only if $T^*\in (B_{e})$ [resp., $(gB_{e})$].
\end{rema}

After examining the previous examples, it is natural to ask the following question: is there a   relation between the classes $(B_{e})$ and $(gB_{e})$? The following theorem  answer this question, in which we prove that the two classes coincide.

\begin{thm}\label{thm1} Let $T\in L(X).$ Then
  $T\in (B_{e})$ if and only if  $T\in (gB_{e}).$
\end{thm}

\begin{proof}   Suppose that $T\in (gB_{e}).$  Since   the equality  $\sigma_{b}(T)= \sigma_{e}(T)\cup \sigma_{d}(T)$ is always true,
then $\sigma_{b}(T)= \sigma_{e}(T)\cup \sigma_{bf}(T)=\sigma_{e}(T).$ So $T\in(B_{e}).$

Conversely, assume that $T\in (B_{e})$ and  let $\lambda\not\in\sigma_{bf}(T)$ be arbitrary. Since $T-\lambda I$ is in particular of topological uniform descent, from the punctured
neighborhood theorem for semi-B-Fredholm operators--Corollary \ref{corollary1}, there exists $\epsilon>0$ such that $D(\lambda, \epsilon)\setminus\{\lambda\}\subset(\sigma_{e}(T))^C=(\sigma_{b}(T))^C.$ By Corollary \ref{corollary0}, we conclude that $\lambda\not\in\sigma_{d}(T).$ Hence  $\sigma_{d}(T)\subset\sigma_{bf}(T)$ and  then $\sigma_{d}(T)=\sigma_{bf}(T).$ This completes the proof.
\end{proof}

The next proposition gives some characterizations of the class $(B_{e})$-operators.

\begin{prop}\label{prop1}
Let $T \in L(X).$ The following statements are
equivalent.\\
(i)  $T\in (B_{e});$\\
(ii)  $\Delta_{e}(T)\subset \mbox{\rm iso}\,\sigma(T)$ {\rm[equivalently  $\mbox{acc}\,\sigma(T)\subseteq \sigma_{e}(T)];$} \\
(iii) $\sigma(T)=\sigma_{e}(T)\cup\mbox{\rm iso}\,\sigma(T);$\\
(iv)  $\Delta_{e}^g(T)\subset \mbox{\rm iso}\,\sigma(T)$ {\rm[equivalently  $\mbox{acc}\,\sigma(T)\subseteq \sigma_{bf}(T)];$} \\
(v) $\sigma(T)=\sigma_{bf}(T)\cup\mbox{\rm iso}\,\sigma(T).$
\end{prop}

\begin{proof}
(i) $\Longleftrightarrow$ (ii) 	The direct implication is clair. Conversely, assume that  $\Delta_{e}(T)\subset \mbox{iso}\,\sigma(T),$  then  $\Delta_{e}(T)= \mbox{iso}\,\sigma(T)\cap(\sigma_{e}(T))^C=\Pi^0(T).$ Thus $T\in (B_{e}).$\\
(ii) $\Longleftrightarrow$ (iii) Assume that $\Delta_{e}(T)\subset\mbox{iso}\,\sigma(T).$ Then  $\sigma(T)=\Delta_{e}(T)\cup\sigma_{e}(T)\subset \mbox{iso}\,\sigma(T) \cup \sigma_{e}(T).$ As the inclusion $\sigma(T)\supset \mbox{iso}\,\sigma(T) \cup \sigma_{e}(T)$ is always true, then   $\sigma(T)=\sigma_{e}(T)\cup\mbox{\rm iso}\,\sigma(T).$ The reverse implication is obvious. In the same way, we show  that (iv) $\Longleftrightarrow$ (v) $\Longleftrightarrow T\in (gB_{e}).$  From Theorem \ref{thm1}, we conclude that the five statements are equivalent.
\end{proof}

\begin{prop}\label{propB}Let $T\in L(X).$ The following statements hold.\\
(i)  $T\in (B_{e});$\\
(ii)  $T\in (B)$ and  $\sigma_{e}(T)=\sigma_{w}(T);$\\
(iii)   $T\in (B)$ and  $\sigma_{bf}(T)=\sigma_{bw}(T).$
\end{prop}
Remark that the class $(B)$ contains the class   $(B_{e})$ as a proper subclass. Indeed,   the right shift operator $R$  belongs to the class $(B),$ since $\sigma_{w}(R)=\sigma_{b}(R)=D(0, 1).$ But it is already mentioned that    $R\not\in(B_{e}).$ Note that here   $\sigma_{e}(R)=C(0, 1)$  and
$C(0, 1)=\sigma_{bf}(R) \neq\sigma_{bw}(R)=D(0, 1).$

To give the reader a precise  view of the subject, we present  another proof of Theorem \ref{thm1} as an immediate consequence  of  Proposition \ref{propB} and Lemma \ref{lem1} below.
\begin{lem}\label{lem1} The next two equivalences hold  for every  $T\in L(X).$ \\
(i)   $\sigma_{e}(T)=\sigma_{w}(T)\Longleftrightarrow \sigma_{bf}(T)=\sigma_{bw}(T).$\\
(ii) $\sigma_{w}(T)=\sigma_{b}(T) \Longleftrightarrow \sigma_{bw}(T)=\sigma_{d}(T).$
\end{lem}
\begin{proof} (i) $\Longrightarrow)$ Let $\lambda\not\in\sigma_{bf}(T)$ be arbitrary. Then from the punctured
neighborhood theorem for semi-B-Fredholm operators--Corollary \ref{corollary1}, there exists $\epsilon>0$ such that for all  $\mu \in D(\lambda, \epsilon)\setminus\{\lambda\},$ we have  $\mu \in (\sigma_{e}(T))^C=(\sigma_{w}(T))^C$ and $\mbox{ind}(T-\mu I)=\mbox{ind}(T-\lambda I)=0.$ Hence $\lambda\not\in\sigma_{bw}(T)$ and then $\sigma_{bw}(T)=\sigma_{bf}(T).$ \\
$\Longleftarrow)$ Since  $\sigma_{w}(T)=\sigma_{e}(T)\cup\sigma_{bw}(T),$ then $\sigma_{w}(T)=\sigma_{e}(T)\cup\sigma_{bf}(T)=\sigma_{e}(T).$\\
(ii) See  \cite[Theorem 2.1]{amouch-zguitti}.
\end{proof}
\begin{rema}
A hole  in $\sigma_{e}(T)$ is a bounded component of $(\sigma_{e}(T))^C.$ It is well known that  every hole $H$   in $\sigma_{e}(T)$ of an operator $T$ is associated with a unique index noted $i(H).$ For more details about this concept see \cite{pearcy}.  Thus  $\sigma_{e}(T)=\sigma_{w}(T)$ if and only if    $\sigma_{e}(T)$ has no holes with non null  index.  For the right shift $R,$ we have $H=\{ \lambda\in\mathbb{C}: 0\leq |\lambda|< 1 \}$ is the unique hole  in $\sigma_{e}(R)=C(0, 1)$  with $i(H)=-1\neq 0.$ This again explains   why $R \notin (B_{e}).$
\end{rema}
It is well known that  if $T\in(aB)$ then $T\in(B),$ but  not conversely, see for example \cite[example: p 406]{bsz}.   Furthermore, we observe that the left shift operator $L\in (aB),$  since $\sigma_{a}(L)=\sigma_{ub}(L)=D(0, 1),$ $\Pi_{a}^0(L)=\emptyset.$  But  $L\not\in(B_{e}),$ since $\sigma_{e}(L)=C(0, 1)$ and $\sigma_{b}(L)=D(0, 1).$ From this and Propsition \ref{propB}, it is of interest to consider the following open question on  an  eventual  relationship  between the classes $(aB)$ and $(B_{e}).$
\begin{itemize}
\item[]{\bf {Question}:} Does there exist an operator $T\in L(X)$ such that $T\in(B_{e}),$  but $T\not\in (aB)$?
\end{itemize}

Now, we  characterize in  the next theorem  the  class  $(B_{e})$-operators in terms of localized SVEP.
\begin{thm}\label{thm2}
Let $T \in L(X).$ The following statements are equivalent.\\
(i) $T\in (B_{e});$\\
(ii) $T\oplus T^{*}$ has the SVEP on $(\sigma_{e}(T))^C;$\\
(iii) $T\oplus T^{*}\in (B_{e});$\\
(iv) $T\oplus T^{*}$ has the SVEP on $(\sigma_{bf}(T))^C.$
\end{thm}
\begin{proof}
(i) $\Longrightarrow$ (ii) It is easy to get    $(\sigma_{e}(T))^C=\Delta_e(T)\cup\, (\sigma(T))^C.$  As $T\in(B_{e}),$  then $$(\sigma_{e}(T))^C = \Pi^0(T) \cup (\sigma(T))^C \subset ({\mathcal S}(T)\cup{\mathcal S}(T^*))^C=({\mathcal S}(T\oplus T^*))^C.$$ So $T\oplus T^{*}$ has the SVEP on $(\sigma_{e}(T))^C.$\\
 (ii) $\Longrightarrow$ (i) If    $T\oplus T^{*}$ has the SVEP on $(\sigma_{e}(T))^C,$ then by implications $(A)$ and $(B)$ mentioned  above, we conclude that  $$(\sigma_{e}(T))^C=(\sigma_{e}(T))^C\cap ({\mathcal S}(T\oplus T^*))^C=(\sigma_{e}(T))^C\cap (\sigma_{d}(T))^C=(\sigma_{b}(T))^C.$$ This proves that  $T\in(B_{e}).$\\
 (i) $\Longleftrightarrow$ (iv) In the same way, we prove that $T\in(gB_{e})$ if and only if
 $T\oplus T^{*}$ has the SVEP on $(\sigma_{bf}(T))^C,$ and since  the class  $(B_{e})$ coincides with the class $(gB_{e}),$ then the proof is complete.\\
(i) $\Longleftrightarrow$ (iii) Clair, since  $\sigma_{e}(T \oplus T^*)= \sigma_{e}(T)\cup \sigma_{e}(T^*)= \sigma_{e}(T)$ and $\sigma_{b}(T \oplus T^*)= \sigma_{b}(T)\cup\sigma_{b}(T^*)=\sigma_{b}(T)$ for every operator $T \in L(X).$
\end{proof}

According to Theorem \ref{thm2}, we have the following remark.

\begin{rema}\label{rema1}~~

\noindent 1.  If $T\in L(X),$ we cannot  guarantee the property $``T\in (B_{e})"$ for an operator $T$  if we restrict  only to the SVEP of $T$  on $(\sigma_{e}(T))^C$ or  to the SVEP of $T^*$ on $(\sigma_{e}(T))^C.$ Indeed,  the right shift operator  $R$ defined above  has the SVEP (since it has no eigenvalues), but  $R\not\in(B_{e}).$  Here  ${\mathcal S}(R\oplus R^*)={\mathcal S}(L)=\{ \lambda\in\mathbb{C}: 0\leq |\lambda|< 1 \}$  and $\sigma_{e}(R)=C(0, 1).$ On the other hand, the left shift operator $L\not\in(B_{e}),$  even if $L^*=R$ has the SVEP.\\
2. We know that $T\in(B_{e})$ if and only if  $T\oplus T^{*}\in (B_{e}).$ Moreover, if  $S,T \in(B_{e})$ then $S\oplus T\in(B_{e}).$  Indeed, $\sigma_{e}(S \oplus T)= \sigma_{e}(S)\cup \sigma_{e}(T)=\sigma_{b}(S)\cup \sigma_{b}(T)=\sigma_{b}(S \oplus T).$ 
\end{rema}

\noindent For $T\in L(X),$ let $\mbox{Hol}(\sigma(T))$ the set of all analytic functions defined on an open
neighborhood of  $\sigma(T),$ and  $\mbox{Hol}_{nc}(\sigma(T))=\{ f\in \mbox{Hol}(\sigma(T)):  f \mbox{ is non-constant on any connected component of } \sigma(T)\}.$

\begin{cor} \label{cor2}Let $T \in L(X)$ and let $f\in\mbox{\rm Hol}(\sigma(T)).$ Then  \\
(i) If $T\in(B_{e})$ and $R \in L(X)$ is a Riesz operator which commutes with $T,$ then $f(T) + R, f(T^*) + R \in (B_{e}).$\\
(ii) If $T\oplus T^*$ has the SVEP and $K \in L(X)$ is an algebraic operator which commutes with $T,$ then $T + K\in(B_{e}).$
\end{cor}

\begin{proof}(i)  Since $T\in (B_{e}),$ then $\sigma_{e}(T)=\sigma_{b}(T)$ and from \cite{Gramsch} we have  $\sigma_{e}(f(T))=\sigma_{b}(f(T)).$ As the essential and Browder spectra of an operator are stable under commuting Riesz perturbations, then
 $\sigma_{e}(f(T)+R)=\sigma_{b}(f(T)+R)$ and $\sigma_{e}(f(T^*)+R)=\sigma_{b}(f(T^*)+R).$ This is equivalent to say that $f(T)+R\in (B_{e})$ and $f(T^*)+R\in (B_{e}).$\\
 (ii)  Since $K$ is algebraic then from \cite[Theorem 2.145]{aiena1}, $T+K$ and $T^*+K^*$ have the SVEP. So $T+K \in (B_{e}).$
\end{proof}

\begin{prop}\label{prop2} For  $T\in L(X),$ the two following statements hold. \\
1. The next assertions for $T$  are equivalent: ~~

i) $T$  has the SVEP on $(\sigma_{e}(T))^C;$ ~~

ii)  $\sigma_{ub}(T)\subset\sigma_{e}(T);$~~

iii) $T$  has the SVEP on $(\sigma_{bf}(T))^C;$ ~~

iv) $\sigma_{ld}(T)\subset\sigma_{bf}(T).$\\
2. $T\in(B_{e})\Longrightarrow\sigma_{ub}(T)\subset\sigma_{e}(T)\Longrightarrow T\in (B).$
\end{prop}

\begin{proof}  ~~

 \noindent 1. \noindent (i) $\Longleftrightarrow$ (ii) Remark that we have always $(\sigma_{e}(T))^C \cap ({\mathcal S}(T))^C=(\sigma_{e}(T))^C \cap (\sigma_{ub}(T))^C$ (see the implication $(A)$ above). If  $T$  has the SVEP on $(\sigma_{e}(T))^C$ then
$(\sigma_{e}(T))^C=(\sigma_{e}(T))^C \cap ({\mathcal S}(T))^C\subset (\sigma_{ub}(T))^C.$ The reverse implication is obvious.\\
(iii) $\Longleftrightarrow$ (iv) Goes similarly with the proof of precedent equivalence. Note that the equality  $(\sigma_{bf}(T))^C \cap ({\mathcal S}(T))^C=(\sigma_{bf}(T))^C \cap (\sigma_{ld}(T))^C$ is always true.\\
(iii) $\Longrightarrow$ (i) Trivial. \\
(ii) $\Longrightarrow$ (iv) Let $\lambda\not\in\sigma_{bf}(T)$ be arbitrary. From the punctured
neighborhood theorem for semi-B-Fredholm operators--Corollary \ref{corollary1}, there exists $\epsilon>0$ such that $D(\lambda, \epsilon)\setminus\{\lambda\}\subset(\sigma_{e}(T))^C\subset(\sigma_{ub}(T))^C.$ By Corollary \ref{corollary0}, we conclude that $p(T-\lambda I)<\infty.$  Furthermore,
  for $n$ large enough, $\R((T-\lambda I)^n)$ is closed, since $T-\lambda I$ is a B-Fredholm. As $\R(T-\lambda I)\cap \NN((T-\lambda I)^{i})=\R(T-\lambda I)\cap \NN((T-\lambda I)^{i+1})$ for all $i\geq p(T-\lambda I),$ we conclude from \cite[Lemma 12]{mbekhta} that $\R((T-\lambda I)^{p(T-\lambda I)+1})$ is closed, and then $\lambda\not\in\sigma_{ld}(T).$ \\
  2. Clair.
\end{proof}

The reverse of  implications proved in  the second statement of Proposition \ref{prop2} are not true, as the next following example shows.

\begin{ex} ~~

\noindent 1.  We consider the left shift operator $L.$  It is clair that $L\in (B),$  but $\sigma_{ub}(L)=D(0, 1)\not\subset\sigma_{e}(L)=C(0, 1).$\\
2. For the right shift operator $R,$ we have  $\sigma_{ub}(R)=\sigma_{e}(R)=C(0, 1),$ but  $R\not\in(B_{e}).$
\end{ex}

According to the Proposition \ref{prop2}, it is natural to ask the following question: let $T\in L(X),$ do we have always $\sigma_{ub}(T)\supset\sigma_{e}(T)\Longleftrightarrow\sigma_{ld}(T)\supset\sigma_{bf}(T)?$ The following theorem answers affirmatively this question.

\begin{thm}\label{thm3} For every  $T\in L(X),$ we have 
$$\sigma_{e}(T)\subset\sigma_{ub}(T)\Longleftrightarrow\sigma_{bf}(T)\subset\sigma_{ld}(T).$$ Furthermore, we have $\sigma_{ub}(T)=\sigma_{e}(T)\Longleftrightarrow\sigma_{ld}(T)=\sigma_{bf}(T).$
\end{thm}
\begin{proof}
$\Longrightarrow$) Let $\lambda\not\in\sigma_{ld}(T)$ be arbitrary, then $T-\lambda I$ is an upper semi-B-Fredholm operator. From the punctured neighborhood theorem for semi-B-Fredholm operators--Corollary \ref {corollary1}, there exists $\epsilon>0$ such that $D(\lambda, \epsilon)\setminus\{\lambda\}\subset(\sigma_{uf}(T))^C$ and from \cite[Theorem 4.7]{Grabiner}, $\mbox{ind}(T-\mu I)= \mbox{ind}(T-\lambda I)$ for all $\mu$ such that $0<|\mu-\lambda|<\epsilon.$ Since $p(T-\lambda I) < \infty$ then  by Corollary \ref{corollary0}, we conclude that  for all $\mu \in D(\lambda, \epsilon)\setminus\{\lambda\},$   $p(T-\mu I)=0$ and  hence  $\mu \not\in \sigma_{ub}(T).$ By hypothesis,  we have $\mu\not\in\sigma_{e}(T).$  Hence $ T-\lambda I$ is an upper semi-B-Fredholm operator of finite index and so  $\lambda \not\in \sigma_{bf}(T).$\\
$\Longrightarrow$) If $\lambda\not\in\sigma_{ub}(T),$ as $\sigma_{ld}(T)\subset\sigma_{ub}(T)$ then by hypothesis  $\lambda \not\in  \sigma_{bf}(T).$ Since $\alpha(T-\lambda I)$ is finite, then  $\beta(T-\lambda I)$ is also finite.  Hence   $\lambda \notin \sigma_{e}(T).$

 As conclusion, we deduce from the previous  equivalence and Proposition \ref{prop2} that $\sigma_{ub}(T)=\sigma_{e}(T)\Longleftrightarrow\sigma_{ld}(T)=\sigma_{bf}(T).$
\end{proof}

We now establish some characterizations of operators belong to the class $(B_{e})$ in terms of the quasinilpotent part $H_{0}(T\oplus T^{*})$ of $T\oplus T^{*}.$ Recall that the \textit{quasinilpotent} part $H_0(T)$ of $T\in L(X)$ is defined
 as the set $H_{0}(T)=\{x\in X: \displaystyle\lim_{n\rightarrow\infty}\|T^{n}(x)\|^{\frac{1}{n}}=0\},$ and the \textit{analytic core} $K(T)$ of $T$ is defined by
$K(T)=\{x\in X : \mbox{ there exist } c>0 \mbox{ and a sequence }
(x_n)  \mbox{ of } X \mbox{ such that } x_0=x, Tx_{n+1}=x_n
 \mbox{ and } \|x_n\|\leq c^n\|x\|
\mbox{ for all } n\in\mathbb{N}\}.$ (See  \cite{aiena1} for information on $H_0(T)$ and $K(T)).$

\begin{thm}\label{thm4}
Let  $T \in L(X).$ The following statements are
equivalent.\\
(i) $T\in (B_{e});$ \\
(ii) For every $\lambda \in \Delta_{e}(T),$   $H_{0}((T\oplus T^{*})-\lambda I)=\mathcal{N} (((T\oplus T^{*})-\lambda I)^{p}),$ where  $p=p(T-\lambda I)\in \N;$\\
(iii) $H_{0}((T\oplus T^{*})-\lambda I)$ has finite-dimensional for every $\lambda \in \Delta_{e}(T);$\\
(iv) $H_{0}((T\oplus T^{*})-\lambda I)$ is closed for all $\lambda \in \Delta_{e}(T);$\\
(v) $K((T\oplus T^*)-\lambda I)$ has finite-codimension for all $\lambda \in \Delta_{e}(T).$
\end{thm}
\begin{proof}
(i) $\Longrightarrow$ (ii) Let $\lambda \in \Delta_{e}(T)$ be arbitrary. As  $T\in(B_{e}),$   then  $\Delta_{e}(T)=\Pi^0(T\oplus T^*).$  From \cite[Corollary 2.47]{aiena1}, $H_{0}((T\oplus T^{*})-\lambda I)=\mathcal{N} (((T\oplus T^{*})-\lambda I)^{p});$ where $p:=p((T\oplus T^{*})-\lambda I).$ Since $\R(T-\lambda I)$ is closed, then $p(T^*-\lambda I)=q(T-\lambda I)$ and hence $p=\max\{p(T-\lambda I), q(T-\lambda I)\}=p(T-\lambda I).$ \\
(ii) $\Longrightarrow$ (iii) If $\lambda\in\Delta_{e}(T),$ then  $\lambda \not\in\sigma_{e}(T)=\sigma_{e}(T\oplus T^*).$ So $\alpha((T\oplus T^{*})-\lambda I)$ is finite and then $\dim H_{0}((T\oplus T^{*})-\lambda I))= \alpha(((T\oplus T^{*})-\lambda I))^p$ is finite. \\
(iii) $\Longrightarrow$ (iv) Obvious.\\
(iv) $\Longrightarrow$ (v) Let $\lambda \in \Delta_{e}(T).$ Since $H_{0}((T\oplus T^{*})-\lambda I)$ is closed, then $T\oplus T^*$ has the SVEP at $\lambda.$ As $(T\oplus T^*)-\lambda I$ is a Fredholm operator, it follows from \cite[Theorem 2.97]{aiena1} that $\lambda \not\in\mbox{acc}\,\sigma_{a}(T\oplus T^*)=\mbox{acc}\,\sigma(T).$ Hence     $\lambda \in \mbox{iso}\,\sigma(T).$ From \cite[Theorem 2.100]{aiena1}, we conclude that  $K((T\oplus T^*)-\lambda I)$ has finite-codimension.\\
(v) $\Longrightarrow$ (i) Let $\lambda \in \Delta_{e}(T).$ Since $K((T\oplus T^*)-\lambda I)\subset\R(((T\oplus T^*)-\lambda I)^n)$ for all $n\in \mathbb{N}$ and   $K((T\oplus T^*)-\lambda I)$ has finite-codimension,  then the sequence $(\R(((T\oplus T^*)-\lambda I)^n))_n$ is stationary, and hence  $q:=q((T\oplus T^*)-\lambda I) <\infty.$ Since $\R(T-\lambda I)$ is closed, then $q(T^*-\lambda I)=p(T-\lambda I)$ and hence $q=\max\{p(T-\lambda I), q(T-\lambda I)\}<\infty.$ Hence $\lambda \in \Pi^
0(T)$ and then  $T\in(B_{e}).$
\end{proof}

Now we give the following theorem, which is a generalization of  the previous theorem to the context of B-Fredholm theory. Its proof goes in a similar way to that of  Theorem \ref {thm4}, and  is left to the reader.

\begin{thm}\label{thm5}
Let  $T \in L(X).$ The following statements are
equivalent:\\
(i) $T\in(gB_{e});$\\
(ii) For every $\lambda \in \Delta_{e}^g(T),$ there exists $p\in \mathbb{N}$ such that  $H_{0}((T\oplus T^{*})-\lambda I)=\mathcal{N} (((T\oplus T^{*})-\lambda I)^{p});$\\
(iii) $H_{0}((T\oplus T^{*})-\lambda I)$ is closed for all $\lambda \in \Delta_{e}^g(T).$
\end{thm}

From Theorem \ref{thm1}, we deduce that  Theorem \ref{thm4} and Theorem \ref{thm5} can be  combined in a single as follows.

\begin{cor}Let $T\in L(X).$ Then the statements of Theorem \ref{thm4} and Theorem \ref{thm5} for $T$  are all equivalent.
\end{cor}

\section{  On the class  $(W_{e})$-operators}

Similarly to the Definition \ref{dfn1}, we introduce in the next  a new class called   $(W_{e})$-operators,  as a proper subclass of the class $(W)$ of operators satisfying Weyl's theorem, and a new class called   $(gW_{e})$-operators, as a proper subclass of the class   $(gW)$ of operators satisfying generalized Weyl's theorem.

\begin{dfn}\label{dfn2}A bounded linear operator $T\in L(X)$ is said to belong to the class $(W_{e})$ [$T\in (W_{e})$ for brevity] if  $\Delta_{e}(T)=E^0(T)$ and is said to belong to the class
    $(gW_{e})$ [$T\in (gW_{e})$ for brevity] if $\Delta_{e}^g(T)=E(T).$
\end{dfn}

\begin{ex}\label{ex3}~~

 \noindent 1. Every normal operator acting on a Hilbert space belongs to the class $(W_{e}),$ see \cite[Chapter XI, Proposition  4.6]{conway}.\\
 \noindent 2. Let $T$   the operator defined on $l^2$ by  $T(x_1,x_2,x_3,\ldots)=(0, \frac{x_1}{2},\frac{x_2}{3}, \frac{x_3}{4},\ldots).$ Then  $\sigma(T)=\sigma_{e}(T)=\sigma_{bf}(T)=\{0\},$ $E(T)=E^0(T)=\emptyset.$  So $T\in (W_{e})\cap(gW_{e}).$\\
  The operator defined on $l^2$  by  $S(x_1,x_2,x_3,\ldots)=(\frac{x_2}{2},\frac{x_3}{3}, \frac{x_4}{4},\ldots),$ does not belong  neither to $(W_{e})$ nor to $(gW_{e}),$ since $E(S)=E^0(S)=\{0\}$ and $\sigma(S)=\sigma_{e}(S)=\sigma_{bf}(S)=\{0\}.$ This examples show also that  the classes $(W_{e})$ and $(gW_{e})$ are not stable under the duality of there elements.  Here $T\in (W_{e})\cap(gW_{e})$ but it's dual $T^*=S \notin (W_{e})\cup(gW_{e}).$\\
\noindent 3.  We consider the operator $U$ defined on $l^2\oplus l^2$ by $U=T\oplus 0,$ where $T$ is the operator defined in the first point. Then  $U \in(W_{e}),$ since $\sigma(U)=\sigma_{e}(U)=\{0\}$ and $E^0(U)=\emptyset.$ But  $U\not\in(gW_{e}),$ since $\sigma_{bf}(U)=E(U)=\{0\}.$
 \end{ex}

 \begin{ex}~~

 \noindent 1. If $T \in L(X)$ is an algebraic operator (equivalently  $\sigma_{d}(T)=\emptyset$), then $T \in (W_{e})\cap(gW_{e}).$ Indeed,  $\Pi(T)=\sigma(T)=E(T)=\sigma_{p}(T)$ and $E^0(T)=\mbox{iso}\,\sigma(T)\cap \sigma_{p}^0(T)=\Pi(T)\cap\sigma_{p}^0(T)=\Pi^0(T)=\sigma_{p}^0(T).$ \\
  2. Let $T \in L(X)$ a Riesz operator. We distinguish two cases.\\
  Case I:  $\sigma(T)$ is  finite

 a) Since $T$ is a Riesz operator, then $\Pi^0(T)=\sigma(T)\setminus\{0\}$ and since $\sigma(T)$ is finite then  $\mbox{iso}\,\sigma(T)=\sigma(T)$ and so $E^0(T)=\sigma_{p}^0(T).$ Hence
    $T \in (W_{e}) \Longleftrightarrow 0 \not\in \sigma_{p}^0(T).$ So if $T\in (W_{e}),$ then    $\sigma(T)\setminus\{0\}=\Pi^0(T)=E^0(T)=\sigma_{p}^0(T).$

 b) As $T$ is a Riesz operator, then either $\sigma_{d}(T)=\emptyset$ or $\sigma_{d}(T)=\{0\}.$   If   $\sigma_{d}(T)=\emptyset,$  then $T\in (gW_{e})$ as mentioned above. And in this case we have $\sigma(T)\setminus\{0\}=\Pi^0(T)=E^0(T)=\sigma_{p}^0(T)$ and $\sigma(T)=\Pi(T)=E(T)=\sigma_{p}(T).$

  If $\sigma_{d}(T)=\{0\},$ then $\Pi(T)=\Pi^0(T)=\sigma(T)\setminus\{0\}$ and $E(T)=\sigma_{p}(T).$ Hence   $T \in (gW_{e}) \Longleftrightarrow 0 \not\in \sigma_{p}(T).$ So if $T\in (gW_{e}),$ then $\sigma(T)\setminus\{0\}=\Pi(T)=\Pi^0(T)=E(T)=E^0(T)=\sigma_{p}(T)=\sigma_{p}^0(T).$\\
 Case II: $\sigma(T)$ is  infinite.\\ In this case, $\mbox{acc}\,\sigma(T)=\{0\}$ and since $\mbox{acc}\,\sigma(T)\subset\sigma_{d}(T),$ it then follows that $\sigma_{d}(T)=\sigma_{b}(T)=\{0\}.$ Hence $\sigma(T)\setminus\{0\}=\Pi(T)=\Pi^0(T)=E(T)=E^0(T).$ Thus $T\in  (gW_{e})\cap  (W_{e}).$
\end{ex}

 We prove in the next proposition that the class $(W_{e})$ (resp., $(gW_{e})$)  is included in  the class $(W)$ (resp., $(gW)$). Moreover,  the right shift operator $R$ defined above shows that these  inclusions are proper.  Indeed, $\sigma(R)=\sigma_{bw}(R)=\sigma_{w}(R)=D(0, 1),$ $\sigma_{bf}(R)=\sigma_{e}(R)=C(0, 1)$ and $E(R)=E^0(R)=\emptyset.$ This proves that $R\in (gW)$ (and so $R\in (W)),$ but $R\not\in (W_{e})\cup (gW_{e}).$
 
\begin{prop}\label{thm6} For  $T\in L(X),$ the following statements hold.\\
(i) $T\in (W_{e})$ if and only if $T\in (W)$ and $\sigma_{e}(T)=\sigma_{w}(T).$\\
(ii) $T\in (gW_{e})$ if and only if $T\in (gW)$ and  $\sigma_{e}(T)=\sigma_{w}(T).$
\end{prop}

\begin{proof} (i) Without restriction on $T,$ it is easy to verify that   $\Delta_{e}(T)\cap E^0(T)\subset \Delta(T)\subset\Delta_{e}(T).$ Since $T\in(W_{e}),$ then $\Delta_{e}(T)=\Delta(T)=E^0(T).$ So   $T\in (W)$  and $\sigma_{e}(T)=\sigma_{w}(T).$ The converse is clear.\\
(ii) Similarly, we have always   $\Delta_{e}^g(T)\cap E(T)\subset \Delta^g(T)\subset\Delta_{e}^g(T).$ As $T\in(gW_{e}),$ then $\Delta_{e}^g(T)=\Delta^g(T)=E(T).$ So  $T\in (gW)$  and $\sigma_{bf}(T)=\sigma_{bw}(T).$ The converse is obvious.
\end{proof}

From Lemma \ref{lem1} and  \cite[Theorem 2.9]{berkani1}, we obtain  immediately the following corollary. Note that the operator $U$ considered in Example \ref{ex3}, shows that   the class  $(W_{e})$ contains the class  $(gW_{e})$ as a proper subclass. Note that $\Pi(U)=\emptyset.$
\begin{cor}\label{cor3} Let $T\in L(X).$ Then
   $T\in (gW_{e})$ if and only if  $T\in (W_{e})$ and $E(T)=\Pi(T).$
\end{cor}

We give in the following proposition, the relationship between the class    $(W_{e})$ [resp., $(gW_{e})$] and  the class $(B_{e})$ [resp., $(gB_{e})$]. And we examine these classes in a particular case of polaroid operators. Recall that an operator $T\in L(X)$  is said to be polaroid (resp., isoloid)  if $\mbox{iso}\,\sigma(T)=\Pi(T)$ (resp., $\mbox{iso}\,\sigma(T)=E(T)$).

\begin{prop}\label{prop3} Let $T\in L(X).$ We have  the following statements.\\
(i) $T\in (W_{e})$ if and only if $T\in (B_{e})$ and $\Pi^0(T)=E^0(T).$\\
(ii)  $T\in (gW_{e})$ if and only if $T\in (B_{e})$ and  $\Pi(T)=E(T).$\\
(iii) If $T$ is a polaroid operator,  then 
 $T\in(W_{e})  \Longleftrightarrow T\in (gW_{e}) \Longleftrightarrow T\in (B_{e}).$
\end{prop}

\begin{proof}
(i) If $T\in(W_{e}),$ then $\Delta_{e}(T)=E^0(T)=\Delta_{e}(T)\cap E^0(T)=\Pi^0(T).$  The converse  is obvious.\\
(ii) Goes similarly with (i).\\
(iii) Remark that if $T$ is a  polaroid operator then  $E(T)=\Pi(T),$ and so $E^0(T)=\Pi^0(T).$ Thus  the desired equivalences.
\end{proof}

\begin{rema} The operator $S$ defined in  Example \ref{ex3}, shows that we do not expect that an operator belonging to the  class $(B_ {e}),$  belongs to the  class $(W_ {e})$ or $ (gW_ {e}).$ Note that $E^0(S)=\{0\}\neq\Pi^0(S)=\emptyset$ and consequently, $E(S)\neq \Pi(S).$ So the class  $(B_{e})$ contains the class  $(W_ {e})$ as a proper subclass which contains the class  $(gW_ {e})$ as a proper subclass.
\end{rema}
\begin{prop}\label{prop4}
Let $T\in L(X)$ be an isoloid operator.  The following statements hold.\\
 (i)   $T\in(gW_ {e})$  if and only if   $f(T)\in (gW_ {e})$ for every    $f\in\mbox{\rm Hol}(\sigma(T)).$\\
 (ii)  $T\in(W_ {e})$   if and only if   $f(T)\in (W_ {e})$ for every    $f\in\mbox{\rm Hol}(\sigma(T)).$
\end{prop}
\begin{proof}
(i) $\Longrightarrow$)  Let $f\in\mbox{\rm Hol}(\sigma(T)).$ Since  $T \in (gW_ {e}),$  that's $\sigma(T)\setminus E(T)=\sigma_{bf}(T).$  So $f\left(\sigma(T)\setminus E(T)\right)=f\left(\sigma_{bf}(T)\right).$ Since $T$ is isoloid then from \cite[Lemma 2.9]{berkani-arroud}, we  have $f\left(\sigma(T)\setminus E(T)\right)=\sigma(f(T))\setminus E(f(T)).$ Hence $\sigma(f(T))\setminus E(f(T))=f\left(\sigma_{bf}(T)\right).$ By \cite[Theorem 3.4]{berkani}, we have $f\left(\sigma_{bf}(T)\right)=\sigma_{bf}(f(T)).$  Consequently $f(T) \in  (gW_ {e}).$\\
 $\Longleftarrow$) By taking $f (\lambda) := \lambda,$ we deduce that $T\in (gW_ {e}).$\\
(ii) From \cite[Lemma 6.51]{aiena1} we have $f\left(\sigma(T)\setminus E^0(T)\right)=\sigma(f(T))\setminus E^0(f(T))$ and by \cite{Gramsch},  $f\left(\sigma_{e}(T)\right)=\sigma_{e}(f(T)).$ So the proof goes similarly with (i).
\end{proof}
It is proved in \cite[Theorem 2.6]{berkani-arroud} that every   hyponormal operator belong to class $(gW).$ But we cannot generalize  this result to the class  $(gW_{e}),$ even in the case   of  pure hyponormal operators. For this,  the right shift operator $R$ is pure hyponormal but it does not belong to the class $(gW_{e}).$ In the following proposition, we give  a necessary and sufficient condition for an hyponormal operator to belong to the class $(gW_{e}),$ and we specify in this case    its  several spectra.  Recall that  an operator $T$ is said to be hyponormal if $\|T^*x\|\leq\|Tx\|$ for all vectors $x.$ An hyponormal $T$ is called pure hyponormal  in case, the only subspace reducing $T$ on which $T$ is a normal operator is the zero subspace. 

\begin{prop}\label{propB}
Let $T$ be   an hyponormal operator acting on a Hilbert space $H.$  Then  $T \in (gW_{e})$ if and only if ${\mathcal S}(T^*)\subset \sigma_{e}(T)\cup(\sigma_{w}(T))^C.$ In addition if $T \in (B_{e}),$ then   $\sigma(T)=\sigma_{s}(T),$ $\sigma_{uf}(T)=\sigma_{uw}(T)=\sigma_{ub}(T),$ $\sigma_{ubf}(T)=\sigma_{ubw}(T)=\sigma_{ld}(T),$    $\sigma_{e}(T)=\sigma_{lf}(T)=\sigma_{w}(T)=\sigma_{lw}(T)=\sigma_{b}(T)=\sigma_{lb}(T)=\mbox{acc}\,\sigma(T)\cup\{\lambda \in \mbox{iso} \,\sigma(T) \text{ such that } \alpha(T-\lambda I) =\infty\}$ and $\sigma_{bf}(T)=\sigma_{lbf}(T)=\sigma_{bw}(T)=\sigma_{lbw}(T)=\sigma_{d}(T)=\sigma_{rd}(T)=\mbox{acc} \,\sigma(T).$  In particular, if $T$ is a pure hyponormal and $T \in (B_{e}),$ then $\sigma(T)=\sigma_{lbf}(T)$ and $\sigma_{a}(T)=\sigma_{ubf}(T).$
\end{prop}
\begin{proof}
Suppose that ${\mathcal S}(T^*)\subset \sigma_{e}(T)\cup(\sigma_{w}(T))^C.$ As $T$ is    hyponormal, then $T \in (W).$ So ${\mathcal S}(T^*)\subset \sigma_{e}(T).$ Moreover,  it is well known that ${\mathcal S}(T)=\emptyset.$  Thus ${\mathcal S}(T \oplus T^*)\subset  \sigma_{e}(T).$ It follows from Theorem \ref{thm2} that $T \in (B_{e}).$ As $T$ is polaroid then  by  Proposition \ref{prop3},  $T \in (gW_{e}).$ The converse is obvious. The rest of the proof is left to the reader as a simple exercise.
\end{proof}
 As an applications of Proposition \ref{propB}, we find  the result of Conway given in    \cite[Chapter XI, Proposition  4.6]{conway} which proves that  every   normal operator belong to class $(W_{e}),$   and    the result of Berkani  given in \cite[Theorem  4.5]{berkani0}.
\begin{cor} If $N$ is a  normal operator acting on a Hilbert space $H$  then  $N \in (gW_{e}).$
 Moreover, we have   $\sigma(N)=\sigma_{ap}(N),$   $\sigma_{e}(N)=\sigma_{uf}(N)=\mbox{acc}\,\sigma(T)\cup\{\lambda \in \mbox{iso} \,\sigma(T) \text{ such that } \alpha(T-\lambda I) =\infty\}$ and $\sigma_{bf}(N)=\sigma_{ubf}(N)=\mbox{acc}\,\sigma(N)=\{\lambda \text{ such that } \R(\lambda I - T) \text{ is not closed}\}.$ 
\end{cor}
\begin{rema} Remark that there exist  hyponormal operators belonging to the class $(gW_{e})$ which are not normal. Let  $N$ be a diagonalisable normal operator all  of whose eigenvalues have infinite multiplicity and such that these eigenvalues are dense in $D(0, 1).$ We consider the operator $T=R\oplus N.$  Then $T$ is an hyponormal   which is  not normal. Moreover,  it is easy to  get ${\mathcal S}(T^*)\subset  \sigma(T)=\sigma_{e}(T)=D(0, 1).$ 
\end{rema}

\noindent {\bf{Conclusion}:}

As  conclusion, we  give a summary of the results obtained in the two preceding parts of
this paper. In the following diagram,
   the arrows signify the relation of inclusion between
   the class $(W),$ the class $(gW),$ the class $(B)$   and various classes introduced and studied in this paper.
 The numbers near the arrows are references to
the results obtained in the present paper (numbers without brackets) or to
the bibliography therein (the numbers in square brackets).

\begin{center}
 \vspace{5pt} \small{
\vbox{
\[
\begin{CD}T\in(W)@<\mbox{{\scriptsize\ref{thm6}}}<<T\in(W_e)
@>\mbox{{\scriptsize\ref{prop3}}}>>T\in(B_e)@>\mbox{{\scriptsize\ref{propB}}}>>T\in(B)\\
@AA\mbox{{\scriptsize\cite{Berkani-koliha}}}A@AA\mbox{{\scriptsize\ref{cor3}}}A\,\,\,\,\,\,\,\,\,\,\Updownarrow\mbox{{\scriptsize\ref{thm1}}}@.
\,\,\,\,\,\,\,\,\,\,\Updownarrow\mbox{{\scriptsize\cite{amouch-zguitti}}}\\
T\in(gW)@<<\mbox{{\scriptsize\ref{thm6}}}<T\in(gW_e)@>>\mbox{{\scriptsize\ref{prop3}}}>T\in(gB_e)@>>\mbox{{\scriptsize\ref{propB}}}>T\in(gB)
\end{CD}
\]}}
\end{center}

Moreover, counterexamples were given to show that the reverse of  each implication  (presented with number without bracket) in
the diagram is not true. Nonetheless, it was proved that under some extra assumptions, these
implications are equivalences.

\goodbreak
{\small \noindent Zakariae Aznay,\\  Laboratory (L.A.N.O), Department of Mathematics,\\Faculty of Science, Mohammed I University,\\  Oujda 60000 Morocco.\\
aznay.zakariae@ump.ac.ma\\

 \noindent Hassan  Zariouh,\newline Department of
Mathematics (CRMEFO),\newline
 \noindent and laboratory (L.A.N.O), Faculty of Science,\newline
  Mohammed I University, Oujda 60000 Morocco.\\
 \noindent h.zariouh@yahoo.fr

\end{document}